\documentclass[12pt, a4paper]{article}
\usepackage{amsfonts}
\usepackage{mathrsfs}
\usepackage{latexsym}
\usepackage{xy}
\usepackage{amsfonts,amsmath,amssymb,amsthm}
\usepackage{color}
\fontsize{11pt}{11pt}\selectfont

\xyoption{all}

\newcommand{\bcen}{\begin{center}}     \newcommand{\ecen}{\end{center}}
\newcommand{\bay}{\begin{array}}      \newcommand{\eay}{\end{array}}
\newcommand{\beq}{\begin{eqnarray*}}      \newcommand{\eeq}{\end{eqnarray*}}

\def\hl{\mbox{\rm hl}}
\def\hr{\mbox{\rm hr}}
\def\hw{\mbox{\rm hw}}

\def\rad{\mathrm{rad}}

\def\dim{\mathrm{dim}}
\def\mod{\mathrm{mod}}

\def\Image{\mathrm{Im}}
\def\Ker{\mathrm{Ker}}

\def\proj{\mathrm{proj}}

\begin{document}

\newtheorem{theorem}{Theorem}
\newtheorem{proposition}{Proposition}
\newtheorem{lemma}{Lemma}
\newtheorem{corollary}{Corollary}
\newtheorem{remark}{Remark}
\newtheorem{example}{Example}
\newtheorem{definition}{Definition}
\newtheorem*{conjecture}{Conjecture}
\newtheorem*{question}{Question}

\title{\large\bf Indecomposables with smaller cohomological length in the derived category of gentle algebras}

\author{\large Chao Zhang}

\date{\footnotesize Department of Mathematics, Guizhou University, Guiyang 550025, P.R. China.\\ E-mail:
zhangc@amss.ac.cn}

\maketitle

\begin{abstract} Bongartz and Ringel proved that there is no gaps in the sequence of
lengths of indecomposable modules for the finite-dimensional algebras over algebraically closed fields.
The present paper mainly study this ``no gaps" theorem as to cohomological length
for the bounded derived category $D^b(A)$
of a gentle algebra $A$: if there is
an indecomposable object in $D^b(A)$ of cohomological length $l>1$, then there exists an indecomposable
with cohomological length $l-1$.
\end{abstract}

\medskip


\medskip


\bigskip

\section{Introduction}

Throughout this paper, $k$ is an algebraically closed field, all
algebras are connected, basic, finite-dimensional, associative
$k$-algebras with identity, and all modules are finite-dimensional
right modules, unless stated otherwise. During the study of the representation
theory of finite-dimensional algebras, the classification and
distribution of indecomposable modules play a significant role.
Besides the famous Brauer-Thrall conjectures \cite{Aus74, Jan57, NR75,Ro68, Ro78},
Bongartz and Ringel proved the
following elegant theorem in \cite{Bo13, R11}:

\medskip

{\bf Theorem 0.} {\it  Let $A$ be a finite-dimensional algebra. If there is
an indecomposable $A$-module of length $n>1$, then there exists an indecomposable
$A$-module of length $n-1$.}
\medskip

Since Happel \cite{Hap88}, the bounded derived categories of
finite-dimensional algebras have been studied widely. The
classification and distribution of indecomposable objects in the
bounded derived category is still an important theme
in representation theory of algebras. In this context, the definitive work was
due to Vossieck \cite{Vo01}. He classified a class of algebras, {\it
derived discrete algebras,} that is, with only finitely many indecomposables distributed in each
cohomology dimension vector in their bounded derived category.
In the research of Brauer-Thrall type theorems for the bounded
derived category of an algebra  \cite{ZH16}, some numerical invariants, i.e.
the cohomological length, width,
and range of a complex in bounded derived category are introduced:
let $A$ be a finite-dimensional algebra with $D^b(A)$ the
bounded derived module category, the {\it cohomological length, cohomological width,
cohomological range} of a complex $X^{\bullet} \in
D^b(A)$ are $$\hl(X^{\bullet}) := \max\{\dim H^i(X^{\bullet}) \; | \;i \in \mathbb{Z}\},$$
$$\hw(X^{\bullet}) := \max\{j-i+1 \; | \; H^i(X^{\bullet}) \neq 0 \neq H^j(X^{\bullet})\},$$
$$\hr(X^{\bullet}) := \hl(X^{\bullet}) \cdot \hw(X^{\bullet}),$$
respectively. Moreover, the derived Brauer-Thrall type theorems are established in \cite{ZH16} with
cohomological range to be the replacement of length of modules in classical Brauer-Thrall conjectures.
Note that there is a
full embedding of $\mod A$ into $D^b(A)$ which sends any $A$-module to the
corresponding stalk complex.
Obviously, the dimension of an $A$-module $M$ is equal to the
cohomological length and the cohomological range of the stalk complex $M$.
As pointed out as a question in \cite{ZH16}, it is natural to
consider the derived version of Bongartz-Ringel's theorem and ask
whether there are no gaps in the sequence of cohomological lengths (ranges) of
indecomposable objects in $D^b(A)$.

\medskip

{\bf Question I} {\it Is there an indecomposable object in $D^b(A)$ of cohomological
length $l-1$ if there is one of cohomological length $l \geq 2$?
}

\medskip

{\bf Question II} {\it Is there an indecomposable object in $D^b(A)$ of cohomological
range $r-1$ if there is one of cohomological range $r \geq 2$?
}

\medskip

Evidently, the questions have positive answers for representation-infinite algebras
by Bongartz-Ringel's theorem for the module category of algebras.
However, it seems difficult to give answers for general finite-dimensional algebras
to above questions since we know little about the description
of indecomposables in the bounded derived category.

In this paper, we prove that for gentle algebras,
the answer to question I is positive,
but the answer of question II is negative. To be precise,
there is no gaps in the sequence of cohomological lengths of
indecomposables in the bounded derived category of gentle algebras.
In addition, we construct a gentle algebra $A_0$ such that there is
an indecomposable object in $D^b(A_0)$ of cohomological range $r_0$ but no
indecomposable object with cohomological range $r_0-1$. Our result relies
on the constructions of indecomposables in the bounded derived category
of gentle algebras due to Bekkert and Merklen \cite{BM03}.

The paper is organized as follows: in Section 2, we shall recall
the constructions of indecomposable objects in the bounded derived category of
gentle algebras. In Section 3, we shall prove the main theorem of this paper.
Finally, we
produce a gentle algebra which demonstrates that Question 2 has a negative answer.

\section{Indecomposables in bounded derived category of gentle algebras}

In this section, we mainly recall the description of the indecomposable objects
in the bounded derived category of gentle algebras from \cite{BM03}.

Let $A$ be an algebra admitting a presentation $kQ/I$ where
$Q$ is a finite quiver with vertex set $Q_0$ and arrow set $Q_1$,
and where $I$ is an admissible ideal of $kQ$. Throughout this paper, we
write the path in $kQ/I$ from left to right.
Recall that $A=kQ/I$ is a {\it gentle algebra} if

{\rm (1)} the number of arrows with a given source (resp. target) is at most two;

{\rm (2)} for any arrow $\alpha\in Q_1$, there is at most one arrow $\beta\in Q_1$
such that $s(\alpha)=t(\beta)$ (resp. $t(\alpha)=s(\beta)$) and $\beta\alpha\in I$
(resp. $\alpha\beta\in I$).

{\rm (3)} for any arrow $\alpha\in Q_1$, there is at most one arrow $\gamma\in Q_1$
such that $s(\alpha)=t(\gamma)$ (resp. $t(\alpha)=s(\gamma)$) and $\gamma\alpha\notin I$
(resp. $\alpha\gamma\notin I$).

{\rm (4)} $I$ is generated by a set of paths of length two.

\medskip

Let $A=kQ/I$ be a gentle algebra. We need to recall some notations.
For a path $p=\alpha_1\alpha_2\cdots\alpha_r$
with $\alpha_i\in Q_1$, we say its {\it length} $l(p)=r$.
Let $\bf{Pa}_{\geq 1}$ be the set of all paths in $kQ/I$ of length
greater than $1$.
For any arrow $\alpha\in Q_1$, we denote by $\alpha^{-1}$
its formal {\it inverse} with $s(\alpha^{-1})=t(\alpha)$ and $t(\alpha^{-1})=s(\alpha)$.
For a path $p=\alpha_1\alpha_2\cdots\alpha_r$, its {\it inverse} $p^{-1}=\alpha_r^{-1}\alpha_{r-1}^{-1}
\cdots\alpha_1^{-1}$. A sequence $w=w_1w_2\cdots w_n$ is a {\it walk}
(resp. a {\it generalized walk}) if
each $w_i$ is of form $p$ or $p^{-1}$ with $p\in Q_1$ (resp. $p\in \bf{Pa}_{\geq 1}$), and
$s(w_{i+1})=t(w_i)$ for $i=1, 2, \cdots, n-1$.

We denote by ${\bf St}$ the set of all walks $w=w_1w_2\cdots w_n$ such that
$w_{i+1}\neq w_i^{-1}$ for each $1\leq i<n$ and no subword of $w$ or $w^{-1}$ lies in $I$. We call
an element in ${\bf{St}}$ a {\it string}. By ${\bf\overline{Gst}}$ we denote the set of all generalized
walks such that

{\rm (1)} $w_iw_{i+1}\in I$ if $w_i, w_{i+1}\in \bf{Pa}_{\geq 1}$;

{\rm (2)} $w^{-1}_{i+1}w_i^{-1}\in I$ if $w^{-1}_i, w^{-1}_{i+1}\in \bf{Pa}_{\geq 1}$;

{\rm (3)} $w_iw_{i+1}\in \bf{St}$ otherwise.

We write $\bf{Gst}$ the set consisting of all trivial paths
and the representatives of ${\bf\overline{Gst}}$ modulo the relation
$w\sim w^{-1}$. An element $w=w_1w_2\cdots w_n$ in $\bf{Gst}$ is called a {\it generalized string}
of width $n$.

Generalized bands are special generalized strings. Before its definition, we
need the following notation. Let $w=w_1w_2\cdots w_n$ be a generalized string,
set $\mu_w(0)=0$, $\mu_w(i)=\mu_w(i-1)-1$ if $w_i\in \bf{Pa}_{\geq 1}$ and
$\mu_w(i)=\mu_w(i-1)+1$ otherwise. Suppose ${\bf \overline{GBa}}$ is the set of all
generalized walks $w=w_1w_2\cdots w_n$ such that

{\rm (1)} $s(w_1)=t(w_n)$;

{\rm (2)} $\mu_w(n)=\mu_w(0)=0$;

{\rm (3)} $w^2=w_1w_2\cdots w_nw_1w_2\cdots w_n\in {\bf \overline{Gst}}$.

We denote by $\bf{Gba}$ the set consisting of
the representatives of ${\bf\overline{Gba}}$ modulo the relation
$w\sim w^{-1}$ and $w_1w_2\cdots w_n\sim w_2\cdots w_nw_1$.
We call an element in $\bf{Gba}$ a {\it generalized band}.

By the description of Bekkert and Merklen \cite{BM03}, a generalized string  in $A=kQ/I$
corresponds to a unique indecomposable object of bounded homotopy category $K^b(\proj A)$
up to shift, while a generalized band $w$ corresponds to a unique family of indecomposables
$\{P^{\bullet}_{w,\lambda} \; | \;
 \lambda \in k^*, \; d>0 \}$ in $K^b(\proj A)$ up to shift, in which
$P^{\bullet}_{w,\lambda}$ and $P^{\bullet}_{w,\lambda'}$
have the same cohomology dimension vector for any $\lambda, \lambda'$. Thus $A$ is derived discrete if and only if
$A$ contains no generalized bands, see \cite{BM03, Vo01}.

Let $\alpha$ be a path in $\bf{Pa}_{\geq 1}$. Then it induces a morphism $P(\alpha)$ from $P_{t(\alpha)}$
to $P_{s(\alpha)}$ by left multiplication, where $P_i$ is the indecomposable projective right
$A$-module $e_iA$ associated to vertex $i$. More precisely,
$P(\alpha)(u)=\alpha u$ for any $u\in kQ/I$.

\begin{definition} Let $w=w_1w_2\cdots w_n$ be a generalized string. Then the complex of projective modules
$P_w^{\bullet}=\xymatrix{
\cdots \ar [r]^{d_w^{i-1}}& P_w^i \ar [r]^{d_w^{i}} &P_w^{i+1} \ar [r]^{d_w^{i+1}} &\cdots}$
is defined as follows. The module on the $i$-th component
$$P_w^i=\bigoplus_{j=0}^n \delta(\mu_w(j),i) P_{c(j)},$$
where $\delta$ is the Kronecker sign, $c(j)=s(w_{j+1})$ for $j<n$ and $c(n)=t(w_n)$.
The differential $d_w^i$ is
given by the matrix $(d_{j,k}^i)$ with entries, where
$$d^i_{j,k}=\left\{\begin{array}{lll} P(w_{j}), & \mbox{if } w_j\in{\bf Pa_{\geq 1}}, \mu_w(j)=i, k=j-1;
\\ P(w_{j+1}^{-1}), & \mbox{if } w_{j+1}^{-1}\in{\bf Pa_{\geq 1}}, \mu_w(j)=i, k=j+1;
\\ 0, & \mbox{otherwise. }
 \end{array}\right.$$
\end{definition}

\begin{definition}
Let $w=w_1w_2\cdots w_n$ be a generalized band. Then for any
$\lambda \in k^*, \; d>0$, the complex of projective modules
$$P_{w,\lambda}^{\bullet}=\xymatrix{
\cdots \ar [r]^{d_w^{i-1}}& P_{w,\lambda}^i \ar [r]^{d_w^{i}} &P_{w,\lambda}^{i+1} \ar [r]^{d_w^{i+1}} &\cdots}$$
is defined as follows. The module on the $i$-th component
$$P_{w,\lambda}^i=\bigoplus_{j=0}^{n-1} \delta(\mu_w(j),i) P_{c(j)}^d.$$
The differential $d_w^i=(d_{j,k}^i)$ and
$$d^i_{j,k}=\left\{\begin{array}{lllll} P(w_{j}){\bf Id}_d, & \mbox{if } w_j\in{\bf Pa_{\geq 1}}, \mu_w(j)=i, k=j-1;
\\ P(w_{j+1}^{-1}){\bf Id}_d, & \mbox{if } w_{j+1}^{-1}\in{\bf Pa_{\geq 1}}, \mu_w(j)=i, k=j+1;
\\ P(w_{n})J_{\lambda, d}, & \mbox{if } w_n\in{\bf Pa_{\geq 1}}, \mu_w(n)=0=i, k=n-1;
\\ P(w_{n}^{-1})J_{\lambda, d}, & \mbox{if } w_{n}^{-1}\in{\bf Pa_{\geq 1}}, \mu_w(n-1)=i, k=0;
\\ 0, & \mbox{otherwise, }
 \end{array}\right.$$ where $J_{\lambda, d}$ the upper
triangular $d \times d$ Jordan block with eigenvalue $\lambda \in k^*$.
\end{definition}
Note that the definitions above are slightly different from ones in \cite{BM03} since
we consider right projective modules throughout this paper.

Recall that a complex $X^{\bullet}=(X^i, d^i) \in C(A)$ is said to
be {\it minimal} if $\Image d^i \subseteq \rad X^{i+1}$ for all $i \in
\mathbb{Z}$. For a complex $P^{\bullet}$ in $C^{-,b}(\proj A)$ of the form $$P^{\bullet}=
\cdots \longrightarrow P^{-n-1}\stackrel{d^{-n-1}}{\longrightarrow}
P^{-n}\stackrel{d^{-n}}{\longrightarrow} \cdots
\longrightarrow P^{m-1}\stackrel{d^{m-1}}{\longrightarrow}
P^m\longrightarrow 0, $$ its {\it brutal truncation} $\sigma_{\geq -n}(P^{\bullet})$ is
$$\sigma_{\geq -n}(P^{\bullet})=0\longrightarrow
P^{-n}\stackrel{d^{-n}}{\longrightarrow} \cdots
\longrightarrow P^{m-1}\stackrel{d^{m-1}}{\longrightarrow}
P^m\longrightarrow 0. $$

The following lemma due to \cite[Proposition 2]{ZH16} sets up the connection between the
indecomposable objects in $K^b(\proj A)$ and those in $K^{-,b}(\proj
A)$.

\begin{lemma} \label{prop-indec}
Let $P^{\bullet} \in K^{-,b}(\proj A)$ be a minimal complex and $-n
:= \min\{i \in \mathbb{Z} \; | \; H^i(P^{\bullet}) \neq 0\}$. Then
$P^{\bullet}$ is indecomposable if and only if so is the brutal
truncation $\sigma_{\geq j}(P^{\bullet}) \in K^b(\proj A)$ for some $j< -n$ or for all some $j< -n$.
\end{lemma}

\medskip

Let $A$ be a finite-dimensional algebra and $P^\bullet\in K^b(\proj A)$ an
indecomposable minimal complex of the form
$$P^{\bullet} = 0 \longrightarrow P^{-n} \stackrel{d^{-n}}{\longrightarrow}
P^{-n+1} \stackrel{d^{-n+1}}{\longrightarrow} \cdots
\stackrel{d^{-2}}{\longrightarrow} P^{-1}
\stackrel{d^{-1}}{\longrightarrow} P^0 \longrightarrow 0.$$
Now we can construct a minimal object in
$D^b(A)$ by eliminating the cohomology of minimal degree. Suppose $H^{-n}(P^\bullet)
\cong \Ker d^{-n}$, we take a minimal projective resolution
of $\Ker d^{-n}$, say
$$P'^{\bullet} = \cdots \longrightarrow P^{-n-2}
\stackrel{d^{-n-2}}{\longrightarrow} P^{-n-1} \longrightarrow 0.$$
Gluing $P'^{\bullet}$ and $P^{\bullet}$ together, we get a minimal
complex
$$\beta(P^\bullet) = \cdots \longrightarrow P^{-n-2}
\stackrel{d^{-n-2}}{\longrightarrow} P^{-n-1}
\stackrel{d^{-n-1}}{\longrightarrow} P^{-n}
\stackrel{d^{-n}}{\longrightarrow} \cdots
\stackrel{d^{-1}}{\longrightarrow} P^0 \longrightarrow 0,$$ where
$d^{-n-1}$ is the composition $P^{-n-1} \twoheadrightarrow \Ker
d^{-n} \hookrightarrow P^{-n}$. Note that $H^{-n}(\beta(P^\bullet))=0$,
and $H^j(\beta(P^\bullet))=H^j(P^\bullet)$ for $j\neq -n$.

\begin{lemma}
\label{cut-coh} Keep the notations as above.
Then  $\beta(P^\bullet)$ is indecomposable.
\end{lemma}

\begin{proof}
If $H^{-n}(P^\bullet)=0$, then $\beta(P^\bullet)=P^\bullet$ and the statement follows.
Now suppose $H^{-n}(P^\bullet)\neq 0$.
Since $P^{\bullet} $ is the brutal truncation $\sigma_{\geq
-n}(\beta(P^{\bullet}))$, which is indecomposable and $H^i(\beta(P^{\bullet})) =0$
for all $i \leq -n$,
$\beta(P^{\bullet})$ is indecomposable by Lemma \ref{prop-indec}.
\end{proof}

The following theorem from \cite[Theorem 3]{BM03} provides an explicit description of
the indecomposables in the bounded derived category $D^b(A)$.

\begin{theorem}
\label{thm-const} Let $A=kQ/I$ be a gentle algebra with $[-1]$ the shift functor in
$D^b(A)$. Then the set of indecomposable objects in $K^b(\proj A)$ is
$$\{P_w^\bullet[i]\;|\; w\in {\bf Gst}, i\in \mathbb{Z}\} \cup
\{P_{w,\lambda}^\bullet[i]\;|\; w\in {\bf Gba},  \; \lambda \in k^*, \; d>0, i\in \mathbb{Z}\}.$$
Moreover, the indecomposables in $K^{-,b}(\proj A)\setminus K^b(\proj A)$ is of the form
$\beta(P_w^\bullet)$ for $w\in {\bf Gst}$ with certain conditions.
\end{theorem}

\medskip

\section{The question I for gentle algebras}

In this section, we will discuss the cohomological lengths of the indecomposables in the bounded derived category
of gentle algebras. Indeed, we prove the following theorem.

\begin{theorem}
\label{thm-no-gap}
Let $A$ be a gentle algebra. If there is
an indecomposable object in $D^b(A)$ of cohomological length $l>1$, then there exists an indecomposable
with cohomological length $l-1$.
\end{theorem}

Before the proof, we need some preparations. First, we recall
the definitions of some numerical invariants for finite-dimensional algebras introduced in \cite{ZH16}.

\begin{definition} Let $A$ be a finite-dimensional algebra with $D^b(A)$ the bounded derived category.
The {\it cohomological length} of a complex $X^{\bullet} \in
D^b(A)$ is $$\hl(X^{\bullet}) := \max\{\dim H^i(X^{\bullet}) \; | \;
i \in \mathbb{Z}\}.$$\end{definition}

As well known, there is a
full embedding of $\mod A$ into $D^b(A)$ which sends an $A$-module $M$ to the
corresponding stalk complex and the
cohomological length of the stalk complex $M$ equals to
dimension of $M$. If $A$
is representation-infinite, i.e., there exist indecomposable
$A$-modules of arbitrary large dimensions,
then the {\it global cohomological length} of $A$
$$\mbox{\rm gl.hl} A := \sup\{\hl(X^{\bullet}) \; | \; X^{\bullet} \in D^b(A)
\mbox{ is indecomposable}\}$$ is infinite.
Moreover, by the Bongartz and Ringel's theorem, Theorem \ref{thm-no-gap} also holds for
representation-infinite algebras since the
Brauer-Trall conjecture I holds in this case \cite{Aus74, Ro68}.

\begin{definition}{\rm The {\it cohomological width} of a complex $X^{\bullet} \in D^b(A)$
is
$$\hw(X^{\bullet}) := \max\{j-i+1 \; | \; H^i(X^{\bullet}) \neq 0 \neq H^j(X^{\bullet})\},$$
and the {\it cohomological range} of $X^{\bullet}$ is
$$\hr(X^{\bullet}) := \hl(X^{\bullet}) \cdot \hw(X^{\bullet}).$$ }\end{definition}

Since the cohomological width of a stalk complex is one, the cohomological range of a
stalk complex is precisely the cohomological length. Thus, there is also no gaps
in the sequence of cohomological ranges of indecomposable objects in $D^b(A)$ if
$A$ is representation-infinite. Moreover, the cohomological length, width and range
are invariant under shifts and isomorphisms.

\medskip

Let $A$ be a gentle algebra.
By Theorem \ref{thm-const}, any indecomposable complex
$P^\bullet\in D^b(A)$ is of the form $P_w^\bullet$
determined by a generalized string $w$, or of the form
$\beta(P_w^\bullet)$ for some generalized string $w$, or of the form
$P^\bullet=P^\bullet_{w,\lambda}$ determined by a generalized band $w$.
Thus we divide the proof of Theorem \ref{thm-no-gap} into three theorems as follows and
their proofs depend strongly on the description of the indecomposables in the
bounded derived category of gentle algebras due to Bekkert and Merklen \cite{BM03}.

We should recall more notations for a gentle algebras $A=kQ/I$ from \cite{BM03, Bo11},
some of which are slightly different for our convenience.
For any $p\in {\bf Pa_{\geq 1}}$, there is a unique maximal path $\tilde{p}=p\hat{p}$
starting with $p$. Besides the path $\tilde{p}$,
there may be another maximal path, say $\check{p}$,
beginning with the starting point $s(p)$ of $p$. If this is not the case, we write $l(\check{p})=0$.
For any walk $p=p_1p_2\cdots p_l$ and any $j<l$, we write $\kappa^{+}_j(p)=p_{j+1}p_{j+2}\cdots p_l$ for the walk
truncating the first $j$ arrows from the path $p$ along the positive direction. Similarly,
we write $\kappa^{-}_j(p)=p_{1}p_{2}\cdots p_{l-j}$ for the walk
truncating the last $j$ arrows from path $p$ along the negative direction.
Moreover, for a path $\alpha$, we denote by $\overline{\alpha}$ the generalized string $\alpha\alpha_1\alpha_2\cdots$ of
maximal width with $\alpha_i\in Q_1$. Note that $\alpha\alpha_1\in I$, $\alpha_i\alpha_{i+1}\in I$ for $i\geq 1$, and
$\overline{\alpha}=\alpha$ if there is no such arrow $\alpha_1$ that $\alpha\alpha_1\in I$.

Now we are ready for the following theorem.

\begin{theorem} Let $A$ be a gentle algebra. If there is an indecomposable
$P_w^\bullet\in K^b(\proj A)$ determined by a generalized string $w$ such that $\hl(P^\bullet)=l>1$, then
there is an indecomposable $P'^\bullet\in D^b(A)$ with $\hl(P'^\bullet)=l-1$.
\end{theorem}

\begin{proof}

We shall divide the proof into two cases.

\medskip

{\bf Case 1:}  Let $w=w_1w_2\cdots w_n$ be a one-sided generalized string, i.e. $w_i\in {\bf Pa_{\geq 1}}$
for all $1\leq i\leq n$, or $w_i^{-1}\in {\bf Pa_{\geq 1}}$ for all $1\leq i\leq n$. Without
loss of generality, we assume $w_i\in {\bf Pa_{\geq 1}}$ for all $1\leq i\leq n$ (Otherwise, we
can consider the generalized string $w^{-1}$, and they determine the same complex).
Let $P^\bullet$ be the complex determined by $w$ of the form
$$P_w^{\bullet}=\xymatrix{ 0 \ar [r]&
P_{t(w_n)} \ar [r]^{P(w_n)}& P_{t(w_{n-1})} \ar [r]^{P(w_{n-1})}
&\cdots \ar [r]^{P(w_2)}&P_{t(w_1)}\ar [r]^{P(w_1)} &P_{s(w_1)} \ar [r]&0, }$$
where $P_{s(w_1)}$ lies in the $0$-th component.  Thus,
$$\begin{array}{rcl} \dim H^0(P_w^\bullet)
&=&\dim P_{s(w_1)}-\dim \Image P(w_1)\\
&=&\dim P_{s(w_1)}-\dim w_1P_{t(w_1)}\\
&=&\big(l(\widetilde{w_1})+l(\check{w_1})+1\big)-\big(l(\widehat{w_1})+1\big)\\
&=&l(w_1)+l(\check{w_1}).\end{array}$$
For any $1\leq i\leq n-1$,
$$\begin{array}{rcl} \dim H^{-i}(P_w^\bullet)&=&\dim \Ker P(w_{i})-\dim \Image P(w_{i+1}) \\
&=&l(\widetilde{w_{i+1}})-\big(l(\widehat{w_{i+1}})+1\big)\\
&=& l(w_{i+1})-1.
\end{array}$$
Similarly, $$\begin{array}{rcl}\dim H^{-n}(P_w^\bullet)&=&\dim \Ker P(w_n)=\#\{p\in{\bf Pa_{\geq 1}}\;|\; w_np=0\}\\
&=&\left\{\begin{array}{ll} 0, & \mbox{if there is no arrows }\alpha \mbox{ such that } w_n\alpha=0;
\\ l(\tilde{\alpha}), & \mbox{if there is an arrow }\alpha \mbox{ such that } w_n\alpha=0.
 \end{array}\right.
\end{array}$$

%Now we can provide a method to induce the dimension of $0$-th cohomology by one in
%terms of generalized strings. Take the generalized string $w$ as above.

Now we suppose $$i=\max\{j\;|\; \dim H^{-j}(P^\bullet_w)=\hl(P^\bullet_w); 0\leq j\leq n\}.$$
We consider the possible values of $i$ in each case.

{\rm (1)} \; If $i=0$, then $\dim H^j(P^\bullet_w)<\dim H^0(P_w^\bullet)$ for
any $j\neq 0$. Now we want to obtain a generalized string which determines a projective
complex whose cohomological length equals to $\dim H^0(P_w^\bullet)-1=l(w_1)+l(\check{w_1})-1$.

If $l(\check{w_1})=0$, namely, $\tilde{w_1}$ is the unique maximal
path starting from $s(w_1)$, then we get a generalized string
$w'=\kappa^+_1(w_1)w_2\cdots w_n$ by the truncating from positive direction.
Now if there is a unique maximal path beginning with $s(w')=s(\kappa^+_1(w_1))$, then
$$\dim H^0(P_{w'}^\bullet)=l(\kappa^+_1(w_1))=l(w_1)-1=\dim H^0(P_w^\bullet)-1,$$
and the cohomologies of
other degrees remain unchanged. Thus $P'^\bullet=P_{w'}^\bullet$ is as required with
$\hl(P'^\bullet)=l-1$. If there is another arrow $p$ starting from $s(w')$ besides $w'$,
then we set $w''=\overline{p}^{-1}\kappa^+_1(w_1)w_2\cdots w_n$.
Indeed,
the complex $P_{w''}^\bullet$ determined by $w''$ can be illustrated as follows
{\tiny $$\xymatrix{ 0 \ar [r]&
P_{t(w_n)} \ar [r]^{P(w_n)}& P_{t(w_{n-1})} \ar [r]^{P(w_{n-1})}
&\cdots \ar [r]^{P(w_2)}&P_{t(w_1)}\ar [r]^{P(\kappa^+_1(w_1))} &P_{s(\kappa^+_1(w_1))} \ar [r] &0  \\
& & &\cdots  \ar [r]^{P(p_1)}&P_{t(p)}\ar [ur]_{P(p)} & &
}$$} with $P_{s(\kappa^+_1(w_1))}$ on the $0$-th component.
Now we calculate the dimension of cohomologies of $P_{w''}^\bullet$.
$$\begin{array}{rcl} \dim H^0(P_{w''}^\bullet)&=&\dim P_{s(\kappa^+_1(w_1))}
-\dim \Image \big(P(\kappa^+_1(w_1)), P(p)\big) \\
&=&l(\widetilde{\kappa^+_1(w_1)})+l(\widetilde{p})+1-\big(l(\widehat{\kappa^+_1(w_1)})+1\big)-\big(l(\widehat{p})+1\big)\\
&=& l(\kappa^+_1(w_1))+l(p)-1=l(\kappa^+_1(w_1))\\
&=& l(w_1)-1=\dim H^0(P_w^\bullet)-1.
\end{array}  \quad (*)$$ Moreover, the cohomologies of
other degrees remain unchanged since $p_i\in Q_1$. Note that if $\overline{p}^{-1}$ is a walk of
infinite length, then $P_{w''}^\bullet$ is of the form $\beta(P_{u}^\bullet)$, where
$u$ is a generalized string obtained by truncation of $w''$ at certain position. So
$P_{w''}^\bullet$ is indecomposable. Thus $P'^\bullet=P_{w''}^\bullet$ is as required with
$\hl(P'^\bullet)=l-1$.

If $l(\check{w_1})=a>0$, then we set $w'=\overline{\check{w_1}}^{-1}w_1w_2\cdots w_n$.
By the calculation as in the equations $(*)$, $\dim H^0(P_{w'}^\bullet)=l(w_1)+l(\check{w_1})-1=\dim H^0(P_w^\bullet)-1,$
and the cohomologies of other degrees remain unchanged. Thus $P'^\bullet=P_{w''}^\bullet$ is the complex as required.

{\rm (2)} \; If $1\leq i\leq n-1$, since $\dim H^i(P^\bullet_w)=l(w_{i+1})-1=\hl(P^\bullet_w)$,
we only need to consider the case $l(w_{i+1})>2$. We set the generalized string
$w'=\kappa^+_2(w_{i+1})w_{i+2}\cdots w_n$ obtained by truncating from the positive direction.
Similar with the discussion in the case (1), if $\kappa^+_2(w_{i+1})$ is the unique maximal path
beginning with $s(\kappa^+_2(w_{i+1}))$,
then $w'$ determines an indecomposable $P_{w'}^\bullet$ such that
$$\begin{array}{rcl}\dim H^{-i}(P^\bullet_{w'}[-i])
&=&\dim H^0(P^\bullet_{w'})=l(\kappa^+_2(w_{i+1}))\\
&=&l(w_{i+1})-2=\dim H^{-i}(P_w^\bullet)-1=\hl(P_w^\bullet)-1,
\end{array}$$
and $\dim H^{-j}(P^\bullet_{w'}[-i])=0$ for any $j< i$,
$\dim H^{-j}(P^\bullet_{w'}[-i])\leq \dim H^{-j}(P_w^\bullet)<\dim H^{-i}(P_w^\bullet)$ for any $j> i$.
So $P_{w'}^\bullet[-i]$ is the complex as required in this case. If there is another arrow $p$ beginning with
$s(\kappa^+_2(w_{i+1}))$, then we set $w''=\overline{p}^{-1}\kappa^+_2(w_{i+1})w_{i+2}\cdots w_n$.
By a similar calculation as in Case (1), $P'^\bullet=P_{w''}^\bullet$
satisfies $\hl(P'^\bullet)=\hl(P_w^\bullet)-1$.

{\rm (3)} \; Finally, for the case
$i=n$, if there is no arrow $\alpha$ such that $w_n\alpha=0$, then $\hl(P_w^\bullet)=0$, which is impossible.
Let $\alpha$ be such an arrow that $w_n\alpha=0$ and $l(\tilde{\alpha})>1$,
then we choose the generalized string $w'=\kappa_1^+(\tilde{\alpha})$.
With a similar discussion as above, if there is a unique path beginning with
$s(w')$, then $w'$ determines the indecomposable object $P_{w'}^\bullet$.
Set the indecomposable object $P'^\bullet=\beta(P_{w'}^\bullet)$, then we have
$\dim H^{-n}(P'^\bullet[-n])=\dim H^{0}(P_{w'}^\bullet)=l(\tilde{\alpha})-1
=\dim H^{-n}(P^\bullet_w)-1=\hl(P_w^\bullet)$, and the cohomologies of
other degrees vanish. Therefore, $\hl(P'^\bullet)=\hl(P^\bullet)-1$.
If there is another arrow $p$ beginning with the starting point of $w'$, then
set $w''=p^{-1}w'=p^{-1}\kappa_1^+(\tilde{\alpha})$ and $P'^\bullet=\beta(P_{w''}^\bullet)$.
Thus $\dim H^{-n}(P'^\bullet[-n])=\dim H^{0}(P_{w''}^\bullet)
=l(\kappa_1^+(\tilde{\alpha}))+l(p)-1=l(\tilde{\alpha})-1
=\dim H^{-n}(P^\bullet_w)-1=\hl(P_w^\bullet)$, and the cohomologies of
other degrees vanish.

In the above three cases, the construction of the indecomposable object $P'^\bullet$ is based on
the generalized string obtained via truncation from the positive direction.
Indeed, in each case, we can also obtain another indecomposable object by
truncating the generalized strings from the negative direction.
We shall take the case (2) above for example.
First, we set $$i=\min\{j\;|\; \dim H^{-j}(P^\bullet_w)=\hl(P^\bullet_w); 0\leq j\leq n\}.$$
Now, we need to reduce the dimension of $i$-th cohomology by $1$ and eliminate the
$j$-th cohomology for $j<-i$.
We get a generalized string $w'=w_1\cdots w_{i}\kappa_1^-(w_{i+1})$ by truncation from
the negative direction. As in the case (1), we glue $w'$ and a generalized string together
if needed to eliminate the cohomology at certain degree. To be precise,
if there is no arrow $\alpha$ such that $\kappa_1^-(w_{i+1})\alpha\in I$,
then $P'^\bullet=P_{w'}^\bullet$ is also an indecomposable object with
$\hl(P'^\bullet)=\hl(P^\bullet)-1$ as required. If there is an arrow $\alpha$ with
$\kappa_1^-(w_{i+1})\alpha\in I$, then we set $w''=w_1\cdots w_{i}\kappa_1^-(w_{i+1})\overline{\alpha}$.
Then by a similar calculation,
$P'^\bullet=P_{w''}^\bullet$ is also an indecomposable object with
$\hl(P'^\bullet)=\hl(P^\bullet)-1$ as required.
Note that in this case, $P'^\bullet=P_{w''}^\bullet=\beta(P_{w'}^\bullet)$.

\medskip

{\bf Case 2: }Let $w=w_1w_2\cdots w_n$ be a generalized string. Without loss of generality,
assume that $w_{1}^{-1}, w_{2}^{-1},\cdots, w_q^{-1}\in{\bf Pa_{\geq 1}}$ and $w_{q+1}, w_{q+2},
\cdots, w_{r}\in {\bf Pa_{\geq 1}}$, while $w_{r+1}^{-1}\in {\bf Pa_{\geq 1}}$.
Then $w$ determines the indecomposable object $P^\bullet_w$ of form
{\tiny $$\xymatrix{ 0 \ar [r]&
P_{s(w_1)} \ar [r]^-{P(w_1^{-1})}&\cdots \ar [r]& P_{s(w_k)}\ar [r]^-{P(w_k^{-1})}&P_{s(w_{k+1})} \ar [r]^-{P(w_{k+1}^{-1})}
& \cdots \ar [r]^-{P(w_{q-1}^{-1})}&P_{s(w_q)}\ar [r]^-{P(w_q^{-1})} &P_{t(w_q)}  \\
 & & & P_{t(w_r)}\ar [r]^{P(w_r)}\ar [dr]_{P(w_{r+1}^{-1})}&P_{t(w_{r-1})} \ar [r]^{P(w_{r-1})}
& \cdots \ar [r]^{P(w_{q+2})}&P_{t(w_{q+1})}\ar [ur]_{P(w_{q+1})} &\\
& & & & P_{s(w_{r+2})} \ar [r]^{P(w_{r+2}^{-1})}&\cdots&&
}$$}

\noindent where $P_{s(w_1)}$ lies in the $0$-th component.

As illustrated above, there may be more than one indecomposable projective direct summands at a component.
Note that at each
component, we can order these indecomposable projective direct summands
{\it which have nonzero cohomology} along the generalized string $w$. For
example, in the above diagram, suppose the projective module $P_{s(w_{k+1})}$ lies in the $i$-th component, then
we write $P_w^i=P^i_1\oplus P^i_2\oplus P^i_3\oplus \cdots$,
where $P^i_1=P_{s(w_{k+1})}$, $P^i_2=P_{t(w_{r-1})}, \cdots$ since the cohomologies are nontrivial at these
direct summands. Then the cohomology of the degree $i$ is the direct summand of
cohomologies at these projective direct summands.

Now, as in Case 1, we want to construct an indecomposable object $P'^\bullet$ such that
$\hl(P'^\bullet)=\hl(P_w^\bullet)-1$. In order to reduce the dimension of cohomologies of $i$-th degree by 1, it
suffices to reduce the dimension of cohomologies at the first projective direct summand of $i$-th degree.
Indeed, we need to find a unique projective direct summand $Q$ satisfying

{\rm 1)} It is the first direct projective summand of its component under the ordering as above;

{\rm 2)} It lies in the $j$-th component such that $\dim H^j(P^\bullet)=\hl(P^\bullet)$;

{\rm 3)} It is the closest one from the starting point along the generalized string
 among those
satisfying {\rm 1)} and {\rm 2)}.

To construct an indecomposable object $P'^\bullet$ such that
$\hl(P'^\bullet)=\hl(P_w^\bullet)-1$,  we only need to construct such $P'^\bullet$
by reducing the dimension of cohomology at $Q$ by 1.
By the analysis in Case 1, we can manage this via
truncating the generalized string from positive or negative side and
gluing suitable generalized string of the form $\overline{p}^{-1}$ or $\overline{p}$  if needed,
except the following two case:

{\rm 1)} $Q$ is the {\it backward turning points} as $P_{t(w_q)}$, i.e.,
$Q=P_{t(w_i)}$ for some $i$ such that $w_i^{-1}, w_{i+1}\in {\bf Pa_{\geq 1}}$.
Let $Q=P_{t(w_i)}$ be a backward turning point. Then
the dimension of cohomology at this point $Q$, write $H^{t(w_i)}(P_w^\bullet)$
(it is unnecessarily the whole cohomology group at this degree)
$$\begin{array}{rcl} \dim H^{t(w_i)}(P_w^\bullet)&=&\dim P(t(w_{i}))-\dim \Image \big(P(w_{i}^{-1}), P(w_{i+1})\big) \\
&=&l(\widetilde{w_{i+1}})+l(\widetilde{w_{i}^{-1}})+1-\big(l(\widehat{w_{i+1}})+1\big)-\big(l(\widehat{w_{i}^{-1}})+1\big)\\
&=& l(w_{i+1})+l(w_{i}^{-1})-1.
\end{array}$$
Set $w'=\kappa^{+}_1(w_i)w_{i+1}\cdots w_n$.
As in Case 1(1), if there is an arrow $p$ such that $\kappa^{+}_1(w_i)p\in I$, then
we write $w''=\overline{p}^{-1}\kappa^{+}_1(w_i)w_{i+1}\cdots w_n$, and $w''=w'$ otherwise.
We have $\dim H^{t(w_i)}(P_{w''}^\bullet)=\dim H^{t(w_i)}(P_w^\bullet)-1$ and then
$\hl(P_{w''}^\bullet)=\hl (P_w^\bullet)-1$.

{\rm 2)} $Q$ is the
{\it forward turning point} as $P_{t(w_r)}$, i.e.,
$Q=P_{t(w_j)}$ for some $j$ such that $w_j, w_{j+1}^{-1}\in {\bf Pa_{\geq 1}}$.
Similarly let $Q=P_{t(w_j)}$ be a forward turning point. Then
the dimension of cohomology at this point
$$ \begin{array}{rcl}\dim H^{t(w_j)}(P_w^\bullet)&=&
\dim \Ker \big(P(w_{j}), P(w_{j+1}^{-1})\big)^T\\
&=&\dim \big(\Ker P(w_{j})\cap\Ker P(w_{j+1}^{-1})\big)\\
&=& 0,\end{array}$$
which is impossible by the choice of $Q$.
\end{proof}

Now we consider the indecomposable objects in $K^{-,b}(\proj A)\setminus K^b(\proj A)$.

\begin{theorem} Let $A$ be a gentle algebra. If there is an indecomposable
$P^\bullet\in K^{-,b}(\proj A)\setminus K^b(\proj A)$ such that $\hl(P^\bullet)=l>1$, then
there is an indecomposable $P'^\bullet\in D^b(A)$ with $\hl(P'^\bullet)=l-1$.
\end{theorem}

\begin{proof}
Since $P^\bullet\in K^{-,b}(\proj A)\setminus K^b(\proj A)$ is indecomposable, by Theorem \ref{thm-const},
the brutal truncation $\sigma_{\geq j}(P^\bullet)\in K^b(\proj A)$ is indecomposable for some
$j\ll 0$, and $\sigma_{\geq j}(P^\bullet)=P_w^\bullet$ for
some generalized string $w$. Now we can
consider the complex $P_w^\bullet$ using the similar argument as Theorem 3.
If $\dim H^j(P_w^\bullet)\leq l$, then $\hl(P_w^\bullet)=l$ and the statement
from the previous theorem . Suppose $\dim H^j(P_w^\bullet)>l$. By a similar analysis in the
proof of previous theorem, we can find a unique projective direct summand $Q$
which satisfies the following: it is the first direct projective summand,
it lies in $m$-th component such that $\dim H^m(P_w^\bullet)=l$
and it is the closest one from the starting point along $w$.
Then we can construct $P_{w'}^\bullet$ by reducing the dimension of cohomology at $Q$ by 1.
Note that $\dim H^j(P_{w'}^\bullet)$ may have the maximal dimension among the cohomologies
of all degrees. If this is the case, then we have an indecomposable object $P_{w''}^\bullet$
obtained by gluing a generalized string to $w'$ to eliminate the cohomology of $j$-th degree as in the proof
of previous theorem and we are done.
\end{proof}

To finish the proof of Theorem \ref{thm-no-gap}, we only need to
prove the last case, i.e. for the indecomposable objects determined by generalized bands.

\begin{theorem} Let $A$ be a gentle algebra. If there is an indecomposable
$P^\bullet\in K^b(\proj A)$ determined by a generalized band $w$ such that $\hl(P^\bullet)=l>1$, then
there is an indecomposable $P'^\bullet\in D^b(A)$ with $\hl(P'^\bullet)=l-1$.
\end{theorem}

\begin{proof}
Let $w=w_1w_2\cdots w_n$ be a generalized band.
We assume without loss of
generality that $w_1^{-1},w_n\in {\bf Pa_{\geq 1}}$  and
$$\mu(0)=\mu(n)=\min \{\mu(i)\; |\; 0\leq i\leq n\}.$$ Then $w$ determines a family of indecomposable objects
$\{P_{w,\lambda}^\bullet\;|\; w\in {\bf Gba},  \lambda \in k^*, \; d>0, i\in \mathbb{Z}\},$ where
$P_{w,\lambda}^\bullet$ has the form of
{\tiny $$\xymatrix@!=1pc{
P_{s(w_1)}^d \ar [rrd]_-{P(w_n)\mathbf{J}_{\lambda,d}} \ar [rr]^-{P(w_1^{-1})
\mathbf{I}_d} && P_{s(w_2)}^d \ar [rr] &&
\cdots  \ar[rr]&& P_{s(w_r)}^d
\ar [rr]^-{P(w_r)\mathbf{I}_d} && P_{t(w_r)}^d \\
&& P_{s(w_n)}^d \ar [rr] && \cdots \ar [rr]&& P_{t(w_{r+1})}^d \ar [rru]_-{P(w_{r+1})\mathbf{I}_d} }$$}
where $P_{s(w_1)}$ lies in the $0$-th component.

By the previous two theorems, it is sufficient to find a generalized string
$w'$ such that $\hl(\beta(P_{w'}^\bullet))=\hl(P_{w,\lambda}^\bullet)$.
We claim the generalized string $w'=(w_1w_2\cdots w_n)^d$ is the one as required.
Roughly speaking, the complex $P_{w'}^\bullet$ can be seen as the one untying the ``band complex" $P_{w,\lambda}^\bullet$
into a ``string complex". Let $P_w^\bullet$ be the indecomposable
object determined by $w=w_1w_2\cdots w_n$ viewed as a generalized string.
Then for any $i\in\mathbb{Z}$ except $i=0$,
$$\dim H^i(P_{w,\lambda}^\bullet)=d\cdot \dim H^i(P_w^\bullet)=\dim H^i(\beta(P_{w'}^\bullet)).$$
Moreover, if $i=0$, then
$$\dim H^0(P_{w,\lambda}^\bullet)=
\dim \big(\Ker P(w_1^{-1}){\bf I_d} \cap \Ker P(w_n) {\bf J_{\lambda,d}}\big)=0=\dim H^0(\beta(P_{w'}^\bullet)).$$
Therefore, $\hl(\beta(P_{w'}^\bullet))=\hl(P_{w,\lambda}^\bullet)$ as claimed.

\end{proof}

\section{A negative answer to question II}

In this section, we will construct a gentle algebra which provides a negative answer to
Question II.

Let $A_0=kQ/I$ be the gentle algebra defined by the quiver
$$\xymatrix{ &1 \ar[d]^{\alpha_1}& & & & \\
 3& 2 \ar [l]_{\alpha_{2}}\ar
[r]^{\alpha_3}& 4 \ar [r]^{\alpha_4} &5 \ar [r]^{\alpha_5} & 6 \ar [r]^{\alpha_6} &7
}$$ and
the admissible ideal generated by $\alpha_1\alpha_3$. Now we consider the
indecomposable object $P_w^{\bullet}$ determined by generalized string $w=\alpha_1$, where
$$P_w^{\bullet}=\xymatrix{
0 \ar [r]& P_2 \ar [r]^{P(w)} &P_1 \ar [r] &0},$$ with $P_1$ in the $0$-th component.
Clearly, $\dim H^{-1}(P^\bullet_w)=4$ and $\dim H^0(P^\bullet_w)=1$. So
$\hr(P^\bullet_w)=\hl(P^\bullet_w)\cdot \hw(P^\bullet_w)=8$.

Next we claim that there is no indecomposable object in $D^b(A_0)$
with cohomological range $7$. Assume to the contrary that there is an indecomposable
$P^\bullet\in K^b(\proj A_0)$ with $\hr(P^\bullet)=7$, then $\hw(P^\bullet)=7$ or $\hl(P^\bullet)=7$.
We shall show they are impossible. Indeed, by the description due to \cite{BM03}, the
indecomposables in the $D^b(A_0)$ are determined by the generalized strings in $A_0$.
Since the indecomposables in $D^b(A_0)$ are determined by the generalized strings, we have
$$\mbox{\rm gl.hw} A_0 := \sup\{\hw(X^{\bullet}) \; | \; X^{\bullet} \in D^b(A_0) \mbox{ is
indecomposable}\}= 3.$$ Moreover, since any generalized string in $A_0$ is one-sided,
each component of the indecomposable object $P^\bullet_w\in
K^b(\mbox{proj} A_0)$ is indecomposable, and then
$$\mbox{\rm gl.hl} A_0 := \sup\{\hl(X^{\bullet}) \; | \; X^{\bullet} \in D^b(A_0) \mbox{ is
indecomposable}\}\leq \dim P_2=6.$$

\medskip

\noindent {\footnotesize {\bf ACKNOWLEDGEMENT.} I would like to thank Yang Han,
for his help and support during my visit in Academy of Mathematics and System Sciences, CAS,
and also for discussions related to this paper. The author is
supported by the National Natural Science Foundation of China (Grant No. 11601098)
and Natural Science Foundation of Guizhou Province (Grant No. QSF[2016]1038).}

\end{document}